\documentclass{article}

\usepackage{amssymb}
\usepackage{amsmath}
\usepackage{amsthm}
\usepackage{hyperref}
\usepackage{authblk}

\newtheorem{theorem}{Theorem}[section]
\newtheorem{lemma}[theorem]{Lemma}
\newtheorem{propo}[theorem]{Proposition}
\newtheorem{claim}[theorem]{Claim}
\newtheorem{corol}[theorem]{Corollary}
\theoremstyle{definition} 
\newtheorem{defin}[theorem]{Definition}
\newtheorem{remark}[theorem]{Remark}

  \newcommand{\ep}{\varepsilon}

\DeclareMathOperator{\last}{end}
\renewcommand{\path}{\mathfrak{A}}

 \newcommand{\GG}{\mathcal{G}}

 \newcommand{\ie}{i.e.\ }

\title{Invariant random perfect matchings in Cayley graphs}

\author[1]{Endre Cs\'oka}
\affil[1]{\small Department of Computer Science\\
\small Eotvos University\\
\small Budapest, Hungary\\
\small\tt csoka@cs.elte.hu}
\author[2]{Gabor Lippner}
\affil[2]{\small Department of Mathematics\\
\small Harvard University\\
\small Cambridge, MA, USA\\
\small \tt lippner@math.harvard.edu}

\begin{document}
\maketitle

\begin{abstract}
We prove that any non-amenable Cayley graph admits a factor of IID perfect matching. We also show that any connected $d$-regular vertex transitive graph admits a perfect matching. The two results together imply that every Cayley graph admits an invariant random perfect matching.

A key step in the proof is a result on graphings that also applies to finite graphs. The finite version says that for any partial matching of a finite regular graph that is a good expander, one can always find an augmenting path whose length is poly-logarithmic in one over the ratio of unmatched vertices. 
\end{abstract}

\section{Introduction}
Let $\Gamma$ be a finitely generated group, and $G$ a locally finite Cayley graph of $\Gamma$. An invariant random subgraph on $G$ is a probability distribution on the set of subgraphs of $G$ that is invariant under the natural action of $\Gamma$ on $G$. 

A factor of IID is a particular way of defining an invariant random subgraph. We only sketch the definition here. First each vertex gets a random number in $[0,1]$,  independently and uniformly. Then each vertex makes a deterministic decision on how the subgraph looks like in its neighborhood, based on what it sees from itself as center. Since each vertex uses the same rule, the distribution of the resulting subgraph is automatically invariant under the action of $\Gamma$. 

Instead of subgraphs, one can also define vertex colorings, or more general structures on $G$. The general name for such a random process is a factor of IID process. An important feature is that such a process can automatically be modeled on any good finite model of $G$. For instance, any factor of IID process on a regular tree can be modeled (with small error) on finite regular graphs with large girth. 

Invariant random processes, and in particular factor of IIDs on Cayley graphs have received considerable attention recently. Standard percolations are trivially factor of IID processes, as well as the free and the wired minimal spanning forests. Another example is the recent solution of the measurable von Neumann's problem by Gaboriau and Lyons (see~\cite{GL}). They show that every non-amenable Cayley graph admits a factor of iid $4$-regular tree. 

It is a long standing open problem to determine the maximum density $i(G)$ of a factor of IID independent subset of a regular tree (mentioned e.g. on the webpage of David Aldous  \footnote{\url{http://www.stat.berkeley.edu/~aldous/Research/OP/inv-tree.html}}). The exact value is unknown, though it is known to be less than $0.46$. Note that trees are bipartite and thus have independent sets of density $1/2$, but the resulting process can not be a factor of IID. The related open question is to determine the limit of the ratio $i(G(n,d))$ of the largest independent subset in $n$ vertex $d$-regular random finite graphs, as $n$ goes to infinity. Bayati, Gamarnik, and Tetali in~\cite{BGT} have recently shown that the limit exists, and the above mentioned modeling phenomenon shows that its value is at least $i(\mathcal{T}_d)$ where $\mathcal{T}_d$ is the $d$-regular infinite tree. A conjecture of Balazs Szegedy (see Conjecture 7.13 in~\cite{HLSz}) claims that this limit is in fact equal to $i(\mathcal{T}_d)$.

In this paper we settle the analogous question for the maximum density of independent edge sets in non-amenable Cayley graphs. An independent edge set in a graph is usually referred to as a matching. An obvious upper bound on the density of a matching is that of the perfect matching, \ie where every vertex is covered by an edge. We show that in our case one can actually achieve the maximum possible density, that is, one can construct a perfect matching as a factor of IID. 

\begin{theorem}\label{theorem:factorofiid}
Let $\Gamma$ be a finitely generated non-amenable group with finite symmetric generating set $S$. Let $G = Cay(\Gamma,S)$ denote the associated Cayley graph. Then there is a factor of IID on $G$ that is almost surely a perfect matching. 
\end{theorem}

This extends the result of Lyons and Nazarov~\cite{LN} who proved the same statement for bipartite non-amenable Cayley graphs. 

In particular, every non-amenable Cayley graph admits an invariant random perfect matching, which was also not known. Jointly with Ab\'ert and Terpai, the authors showed the following theorem. 

\begin{theorem}\label{theorem:existsmatching} Every infinite vertex transitive graph $G$ has a perfect matching.
\end{theorem}

Ab\'ert and Terpai kindly suggested to include the result in this paper. Now, following an observation of Conley, Kechris, and Tucker-Drob (\cite{CKT}) this implies that every amenable Cayley graph admits an invariant random perfect matching. Thus, together with Theorem~\ref{theorem:factorofiid} we get the following. 

\begin{corol} Every Cayley graph admits an invariant random perfect matching.
\end{corol}

The basic strategy of the proof of Theorem~\ref{theorem:factorofiid} is similar to what Lyons and Nazarov use to prove the bipartite case, and what has been used by Elek and Lippner \cite{EL} to construct almost-maximal matchings. We define a sequence of partial matchings, each of which is obtained from the previous one by flipping a sequence of augmenting paths. To show that this sequence "converges" to a limit perfect matching, one has to show that edges do not change roles too often. The crucial step is to bound the length of the shortest augmenting path in terms of the ratio of unmatched vertices. 

Our main contribution is establishing this bound for general graphs. When applying the result on finite graphs, we get the following theorem, that is of independent interest in computer science. 

\begin{theorem}\label{theorem:finite_shortalternating} 
For any $c_0 > 0$ and $d \geq$ 3 integer, there is a constant $c =
c(c_0,d)$ that satisfies the following statement. If a partial matching in a
$c_0$-expander $d$-regular graph leaves at least $\ep$ ratio of all vertices unmatched, there is an augmenting path of length at most $c \log^3(1/\ep)$, or there is a set of vertices $H \subset G$ such that $|H| \geq 3$, $|H|$ is odd, and the number of edges leaving $H$ is at most $d$.
\end{theorem}

\begin{remark} The theorem remains true even if there are only two unmatched vertices. This may be surprising at first, but in fact the condition that any odd set $H$ has at least $d$ edges leaving it easily implies the conditions of Tutte's theorem, so such graphs always have perfect matchings.  
\end{remark}

In the bipartite case, such a bound has actually already been observed in~\cite{JV} by Jerrum and Vazirani, who used it to give a sub-exponential approximation scheme for the permanent. They remark in the same paper that a similar bound for general graphs would be desirable, as it would lead to an approximation scheme for the number of perfect matchings for arbitrary graphs. In a subsequent paper we shall work out the details of this application, together with a generalization of Theorem~\ref{theorem:finite_shortalternating} to non-regular graphs.

The outline of the paper is as follows. In Section~\ref{section:existence} we show the existence of perfect matchings in vertex transitive graphs. In Section~\ref{section:factorofiid} we prove that in a non-amenable Cayley graph there is a factor of IID that is a perfect matching, modulo a variant of Theorem~\ref{theorem:finite_shortalternating}, whose proof we postpone to Section~\ref{section:shortaugmenting}. 
\paragraph{Acknowledgements.} We are indebted to Mikl\'os Ab\'ert for introducing the problem to us and for his constant encouragement. We would also like to thank him and Tam\'as Terpai for valuable discussions and their kind permission to include the proof of Theorem~\ref{theorem:existsmatching} in this paper.

Both authors' research is partially supported by the MTA Renyi "Lendulet" Groups and
Graphs Research Group. Gabor Lippner is further supported by AFOSR grant FA9550-09-1-0090-DOD-35-CAP.


\subsection{Notation and definitions}

Let $G$ be a simple graph, either finite or infinite. The vertex and edge set of $G$ will be denoted by $V(G)$ and $E(G)$ respectively.

\begin{defin} A \textit{matching} in $G$ is a subset $M \subset E(G)$ such that any vertex $x$ is adjacent to at most one edge $e \in M$. We will denote by $V(M)$ the set of vertices that are matched, \ie that are adjacent to an edge in $M$. 
A matching is \textit{perfect} if $V(M) = V(G)$.
\end{defin}

\begin{defin} Given a graph $G$ with a matching $M$, an \textit{alternating path} is a path $x_0 x_1 \dots x_k$ in $G$ such that every second edge belongs to $M$. An alternating path is called an \textit{augmenting} path if its first and last vertices are not matched. 

If $x,y \in V(G)$ are unmatched vertices and $p$ is an augmenting path connecting $x$ and $y$, then we can define a new matching $M' = M(p) = M \circ E(p)$ as the symmetric difference of the old matching $M$ and the set of edges of $p$. The new matching will then satisfy $V(M') = V(M) \cup \{x,y\}$.
\end{defin}

Let $(X,\mu)$ be a standard Borel probability measure space with a non-atomic probability measure $\mu$. 

\begin{defin}\label{def:graphing} A \textit{graphing} on $X$ is a graph $\GG$ such that $V(\GG) = X$, and where $\GG(E) \subset X \times X$ is a symmetric measurable subset, such that if $A, B \subset X$ are measurable subsets and $f : A \to B$ a measurable bijection whose graph $\{(x,f(x)) : x\in A\}$ is a subset of $E(\GG)$, then $\mu(A) = \mu(B)$.

There is a natural way to measure the size of edge sets in a graphing. If an edge set is given by a measurable bijection $f  : A \to B$ as before, then the size of this edge set is defined to be $\mu(A)$. This extends to a measure on all measurable edge sets. In particular this implies that if $\mathcal{H}$ is a sub graphing of $\GG$ then the size of the edge set of $\mathcal{H}$ can be computed by the formula
\begin{equation}\label{formula:measure} |E(\mathcal{H})| = \frac{1}{2}\int_{X} \deg_{\mathcal{H}}(x) d\mu(x).\end{equation}

A measurable matching (or matching for short) in $\GG$ is a measurable subset $M \subset E(\GG)$ such that every vertex is adjacent to at most one edge in $M$. A matching is \textit{almost everywhere perfect} if $\mu(V(\GG) \setminus V(M)) = 0$. 
In this paper we will only be interested in almost everywhere perfect matchings, and will refer to them as perfect matchings for short.

A graphing $\GG$ is a $c_0$-expander if for every measurable set $H \subset V(\GG)$ we have $|E(H, V(\GG) \setminus H)| \geq c_0 |H| |V(\GG) \setminus H|$, where $E(A,B)$ denotes the set of edges having one endpoint in $A$ and one endpoint in $B$. 
\end{defin}

Let $\Gamma$ be a finitely generated group, and $S \subset \Gamma$ a finite symmetric generating set, and $G = Cay(\Gamma, S)$ the associated Cayley graph, that is $g \in \Gamma$ is connected to $gs$ for every $s \in S$.
$\Gamma$ acts on itself by left multiplication, and this naturally extends to a left action on $X = [0,1]^{V(G)} = [0,1]^\Gamma$ by $gx(\gamma) = x(g^{-1}\gamma)$. The latter action is called the Bernoulli shift of $\Gamma$. We can equip $X$ with a probability measure $\mu$ which is the product of the Lebesgue measure in each coordinate. It is easy to see that the Bernoulli shift action is measure preserving.

$\Gamma$ also naturally acts from the left on $Y = \{0,1\}^{E(G)}$ whose elements can be considered as subsets of $E(G)$. We can also equip $Y$ with the product of uniform measures on the coordinates.

\begin{defin} In our context a \textit{factor of IID} is a measurable, $\Gamma$ equivariant map $\phi : X \to Y$.
\end{defin}

\begin{defin}\label{def:bernoulligraphing} The graphing $\GG$ associated to the Bernoulli shift and $S$ is given by $\GG(V) = X$ and $\GG(E) = \{(x,y) \in X \times X: \exists s \in S, s^{-1}(x) = y\}$. The connected component of almost any point $x\in X$ is isomorphic to the Cayley graph $G$.
\end{defin} 

\begin{claim}\label{claim:correspondence} There is a one-to-one correspondence between measurable subsets $F \subset E(\GG)$ and factors $\phi : X \to Y$. 
\end{claim}

\begin{proof} Let $F \subset E(\GG)$ be a measurable subset and $f : E(\GG) \to \{0,1\}$ its characteristic function. Define $\phi_F : X \to Y$ by the following formula.
\[ \phi_F(x)(g,gs) = f(s^{-1}g^{-1}x,g^{-1}x). \] 
Then
\[ (h\phi_F(x))(g,gs) = \phi_F(x)(h^{-1}g, h^{-1}gs) = f(s^{-1}g^{-1}hx, g^{-1}hx) = \phi_F(hx)(g,gs),\] so we do get a factor.

Conversely, given a factor $\phi$ one can define a subset $F_\phi \subset E(\GG)$ by choosing the edge $s^{-1}x,x$ to be part of $F_\phi$ if and only if $\phi(x)(id,s) = 1$.
\end{proof}

\begin{remark}\label{remark:aepmatching} ~
\begin{itemize}
\item From this construction it is clear that $F$ is an almost everywhere perfect matching if and only if $\phi$ is a factor of IID perfect matching.
\item There is an entirely analogous correspondence between measurable subsets of $V(\GG)$ and factors $\phi : X \to \{0,1\}^\Gamma$. In Lemma 2.3 of~\cite{LN} then translates into the fact that if the Cayley graph $G$ is non-amenable then there is a $c_0 > 0$ depending only on the expansion of $G$, such that the graphing $\GG$ associated to the Bernoulli shift is a $c_0$-expander.
\end{itemize} 
\end{remark}


\section{Perfect matchings in vertex transitive graphs}\label{section:existence}

Let $G(V,E)$ be an infinite, connected, $d$-regular, vertex transitive graph. In this section we show that $G$ has a perfect matching. The proof is done in three steps. 

\begin{defin} A \textit{cut} is a partition of $V$ into a nonempty finite set $A$ and its complement $A^c = V \setminus A$. The size of the cut is the number of edges between $A$ and its complement.  A \textit{best cut} is a cut with minimum size. 
\end{defin}

\begin{lemma} Suppose $A,B \subset V$ are different finite subsets defining best cuts. Then each of the sets $A \setminus B$, $B \setminus A$, $A \cup B$, and $A \cap B$ is either empty or defines a best cut.
\end{lemma}

\begin{proof} Let $X = A \setminus B, Y = B \setminus A, Z = A \cap B, W  = V \setminus (A \cup B)$. Then 
\begin{multline*} |E(X, X^c)| + |E(Y, Y^c)| =\\= 2|E(X,Y)| +|E(X,Z)|+|E(X,W)|+|E(Y,Z)|+|E(Y,W)| \leq \\ \leq 2|E(X,Y)| +|E(X,Z)|+|E(X,W)|+|E(Y,Z)|+|E(Y,W)| + 2|E(Z,W)| = \\ = |E(X \cup Z,Y\cup W)| + |E(Y \cup Z, X \cup W)| = |E(A, A^c)| + |E(B,B^c)|
\end{multline*}
This shows that the cuts defined by $X$ and $Y$ are together at most twice the size of the best cut, hence they must be best cuts as well. (Or empty sets.) A similar argument works for $Z$ and $W$ (or rather $X \cup Y \cup Z$, since that is the finite set) as well.
\end{proof}

\begin{lemma}\label{lemma:smallcut} The size of the best cut in $G$ is $d$.
\end{lemma}
\begin{proof}
Let $X$ be a smallest finite set that defines a best cut. For any pair of vertices $x,y \in X$ there is an automorphism of $G$ that takes $x$ to $y$. Let $Y$ be the image of $X$ under this automorphism. Then clearly $Y$ also defines a best cut, hence $X \setminus Y$ is also a best cut. But $|X \setminus Y| < |X|$ contradicting the minimality of $X$, unless $X = Y$. Hence the graph spanned by $X$ is vertex transitive. 
If $|X| < d$, then the number of edges leaving $X$ is at least $|X|(d - |X| +1) \geq d$ and we are done. If $|X| \geq d$, then since $G$ is connected, there is an edge between a vertex $x \in X$ and $V \setminus X$. But then by vertex transitivity of $X$, there is such an edge from every single vertex of $X$, giving the desired lower bound $|E(X,X^c)| \geq |X| \geq d$.
\end{proof}

\begin{corol} Since the number of edges leaving any finite set $Y$ is at most $d|Y|$, and the number of edges entering any finite set $X$ is at least $d$, we get that the number of finite of components of $G \setminus Y$ is at most $|Y|$. 
\end{corol}

Now we are ready to show the existence of perfect matchings in infinite vertex transitive graphs.

\begin{proof}[Proof of Theorem~\ref{theorem:existsmatching}] By compactness it is sufficient to show that any finite subset $X \subset V$ can be covered by a matching in $G$. So assume for contradiction that there is no matching in $G$ that covers a given finite set $X$. 

Let us construct an auxiliary finite graph $G'(V',E')$ as follows. Let $V' = X \cup \partial X \cup M$ where $\partial X$ is the outer vertex boundary of $X$ and $M$ is a non-empty set of new vertices such that $|V'|$ is even. We define the edge set $E'$ to contain all original edges spanned by $X \cup \partial X$, furthermore we add all edges in $\partial X \cup M$ to make it a clique. 

If $G'$ has a perfect matching, then just keeping those edges of the matching that intersect $X$ gives a matching in $G$ that covers $X$. So we can assume that $G'$ does not have a perfect matching. Then by Tutte's theorem there is a set $Y \subset V'$ such that the number of odd components of $G' \setminus Y$ is greater than $|Y|$. But since $|V'|$ is even, we actually get that the number of odd components of $G' \setminus Y$ is at least $|Y| + 2$. 

The vertices of $\partial X \cup M$ are always in a single component. Thus we can assume that $Y$ is disjoint from $M$, since removing vertices of $M$ from $Y$  affects at most one component while reducing the size of $Y$. Then $Y$ can be thought of a subset of $V$, and it is easy to see that any finite component of $G' \setminus Y$ is also a finite component of $G \setminus Y$, except perhaps for the one component containing $M$. Still, this means that $G \setminus Y$ has at least $|Y| + 1$ odd components, all of which are finite, contradicting the previous Corollary.
\end{proof}

Later we will need a slight strengthening of Lemma~\ref{lemma:smallcut}. We say that a \textit{real cut} is a cut where the finite set has at least 2 elements. 

\begin{lemma}\label{lemma:strongsmallcut} The size of the smallest real cut is bigger than $d$, unless every vertex of $G$ is in a unique $d$-clique.
\end{lemma}

\begin{proof} Suppose the size of the smallest real cut is $d$, and let $X$ be a smallest finite set that defines a smallest real cut. It is clear that $|X| > 2$ since a set of size 2 defines a cut of size at least $2d - 2 > d$. As before, let $x,y \in X$ and let $Y$ be the image of $X$ under an automorphism taking $x$ to $y$. We are going to distinguish between three cases according to the size of $X \setminus Y$, which is the same as the size of $Y \setminus X$.

If they have more than 1 element each, then they also real cuts and hence by Lemma~\ref{lemma:smallcut} they are also smallest real cuts, contradicting the minimality of $X$. 

If  they are both of size 1, then $|X \cap Y|$ and $|X \cup Y|$ both have to be bigger than 1, hence they are also smallest real cuts, again contradicting the minimality of $X$. 

Thus $|X \setminus Y| = 0$, hence $X = Y$, so just like in the proof of Lemma~\ref{lemma:smallcut} we get that $X$ itself is vertex transitive. Thus, by connectivity, each vertex of $X$ has an edge leaving $X$. Thus if $|X| \geq d+1$ then we are done. So $|X| \leq d$ and thus the number of edges leaving $X$ is at least $|X|(d - |X| +1)$. This is strictly greater than $d$, unless $|X| = 1$ or $X$ is a clique of size $d$. The first is clearly not the case since $X$ is a real cut. Thus $X$ is a $d$-clique. Then, of course, by transitivity every vertex of $G$ is in a $d$-clique.  

Finally, it is not possible that a vertex is contained in more than one $d$-clique. If two different $d$-cliques $A$ and $B$ intersect then by degree of the vertices in the intersection we see that $|A \cap B| = d-1$. Let $\{a\} = A \setminus B$ and $\{b\} = B \setminus A$. If $a$ and $b$ would be neighbors then the graph would not be connected. Thus $a$ has to have one neighbor $c$ outside of $B$.  But $c$ cannot be connected to vertices in $A \cap B$, so $A$ is the only $d$-clique that contains $a$. But by transitivity each vertex has to be contained in the same number of $d$-cliques, contradicting our setup. Thus two different $d$-cliques cannot intersect.
\end{proof}

\begin{corol}\label{corol:clique_structure}
If the size of the smallest real cut in $G$ is exactly $d$ then there is a perfect matching in $G$ that is invariant under the automorphism group of $G$. This matching is given by choosing the unique edge from each vertex that leaves the $d$-clique the vertex is contained in. 
\end{corol}


\section{Factor of iid perfect matchings via Borel graphs -- the proof of Theorem~\ref{theorem:factorofiid}}\label{section:factorofiid}

Let $\Gamma$ be a finitely generated non-amenable group, $S$ a finite symmetric generating set of size $|S|=d$, and $G$ the associated Cayley graph. We want to construct a factor of IID perfect matching in $G$.

If the size of the smallest real cut in $G$ is equal to $d$, then by Corollary~\ref{corol:clique_structure} there is a fixed perfect matching in $G$ that is invariant under the action of the automorphism group, and each vertex can decide which edge to choose by observing its own 1-neighborhood, so this is clearly a factor of IID matching and we are done.

Thus we can assume that the smallest real cut in $G$ is at least of size $d+1$. Let $\GG$ be the graphing associated to the Bernoulli shift, as in Definition~\ref{def:bernoulligraphing}. By Claim~\ref{claim:correspondence} and Remark~\ref{remark:aepmatching} it follows that $\GG$ is a $c_0$-expander for some $c_0 > 0$ depending only on $G$. Hence $\GG$ is admissible in the sense of Definition~\ref{conditions}. 

 By Remark~\ref{remark:aepmatching} it is now sufficient to prove that $\GG$ has an almost everywhere perfect matching.  Proposition 1.1 in~\cite{EL} shows that there exists a sequence of matchings $M_0, M_1, M_2, \dots \subset \GG$ such that a) there are no augmenting paths of length $2k+1$ in $M_k$ and b) each $M_{k}$ is obtained from $M_{k-1}$ by a sequence of flipping augmenting paths of length at most $2k+1$. We would like to construct an almost everywhere perfect matching as a limit of the $M_k$s. In order to do this, we have to show that, except for a measure zero set, the status of any edge changes only finitely many times during the process, so we can take a "pointwise" limit of the sequence to obtain a matching that covers but a zero measure subset of $X$.

Let us denote by $U_k$ the set of unmatched vertices in $M_k$. Then in the process of getting $M_{k+1}$ from $M_k$ we are flipping augmenting paths starting and ending in $U_k$. Furthermore each vertex of $U_k$ can be only used once as an endpoint of an augmenting path, since after that it becomes a matched vertex. Any edge that changes status between $M_k$ and $M_{k+1}$ has to be part of an augmenting path at least once. Thus the total measure of status changing edges in this step is at most $(2k+3)|U_k|$. If we can show that $\sum_k (2k+3)|U_k| < \infty$ then by the Borel-Cantelli lemma the measure of edges that change status infinitely many times is zero, and we are done. 

We have seen that $\GG$ is admissible. Let $\ep = |U_k|$. Then by Theorem~\ref{shortalternating_theorem} there is a constant $c = c(c_0,d)$ depending only on the expansion of $\GG$ and the degree $d$, such that there is an augmenting path of length at most $c \log^3(1/\ep)$ in $M_k$. But by definition we know that this has to be longer than $2k+1$. Thus we get $2k+1 \leq c \log^3(1/\ep)$ or equivalently 
\[ |U_k| = \ep < \exp\left(-\left(\frac{2k+1}{c}\right)^{1/3}\right).\]
This is clearly small enough to guarantee that $\sum_k (2k+3)|U_k| < \infty$ and thus complete the proof of Theorem~\ref{theorem:factorofiid}. \qed

\begin{corol} Every $d$-regular infinite Cayley graph has an invariant random perfect matching.
\end{corol}

\begin{proof} For amenable graphs Conley, Kechris and Tucker-Drob observed in Proposition 7.5 of~\cite{CKT} that Theorem~\ref{theorem:existsmatching} implies the existence of invariant random matchings. 

Since a factor of IID perfect matching is automatically an invariant random perfect matching, Theorem~\ref{theorem:factorofiid} completes the non-amenable case.
\end{proof}


\section{Short alternating paths in expanders}\label{section:shortaugmenting}

Let $G(X,E)$ be a $d$-regular graphing, or a connected, $d$-regular graph that can either be finite or infinite. We are going to treat these three cases at the same time. When it is necessary to point out differences, we will refer to them as the measurable/finite/countable case respectively.

\begin{defin}\label{conditions}
We say that $G$ is \textit{admissible} if it is a $c_0$-expander, and the smallest real cut into odd sets has size at least $d+1$ (in the sense of Lemma~\ref{lemma:strongsmallcut}).
\end{defin}

The following theorem includes the statement of Theorem~\ref{theorem:finite_shortalternating} and the variant about graphings that is needed for the proof of Theorem~\ref{theorem:factorofiid}.

\begin{theorem}\label{shortalternating_theorem} 
For any $c_0 > 0$ and $d \geq$ 3 integer, there is a constant $c =
c(c_0,d)$ that satisfies the following statement. Given any
admissible measurable (or large finite) graph, and a partial
matching with at least $\ep$ measure (or fraction) of unmatched
vertices, there is an augmenting path of length at most $c \log^3(1/\ep)$.

\end{theorem}

Though our main goal is to prove theorems about measurable graphs and finite graphs, we are going to need auxiliary results about infinite, connected $d$-regular graphs as well. Since the three cases can be handled the same way, we are going to present the proofs at the same time, pointing out differences when necessary. In the measurable case, everything will be assumed to be measurable, unless explicitly stated otherwise. If $A, B \subset X$ then $E(A,B)$ will denote the set of edges that have one endpoint in $A$ and the other in $B$. In the measurable case the measure of the set $A$ will be denoted by $|A|$. In the finite case $|A|$ is going to denote the size of $A$ divided by the total number of vertices in the graph. So in both of these cases $0 \leq |A| \leq 1$. In fact, a finite graph can be considered as a graphing with an atomic probability measure. However in the countable case $|A|$ is going to simply denote the size of $A$.  Similarly with edge sets, in the finite and the measurable cases $|E(A,B)|$ will denote the measure of the edge set as defined by the integral (\ref{formula:measure}) in Definition~\ref{def:graphing}. In the countable case $|E(A,B)|$ will just denote the size of the set $E(A,B)$. If we really want to talk about the actual size of sets in the finite case, we will denote it by $||A||$ and $||E(A,B)||$ respectively. 

Let $M \subset E$ be a matching. Then $V(M) \subset X$ shall denote the set of matched vertices. Let $S \subset X \setminus V(M)$ denote a fixed subset of the unmatched vertices and let $F = X \setminus (V(M) \cup S)$ denote the remaining unmatched vertices. We are going to construct alternating paths starting from $S$ in the hope of finding an alternating path connecting two unmatched vertices. Such an alternating path is called an \textit{augmenting path}.  

\subsection{Sketch of the proof}

First we give an outline, pointing out the main ideas without introducing the technical definitions. We encourage the reader to read the whole outline before reading the proof, and also to refer back to it whenever necessary. Without understanding the basic outline, many technical definitions will likely be rather unmotivated. 

\begin{enumerate}
\item We start from a set of unmatched vertices $S$. Assuming there are no short augmenting paths, we would like to show that the set of vertices ($X_n$) accessible via $n$-step or shorter alternating paths grows rapidly, eventually exceeding the size of the whole graph, leading to a contradiction.
\item It will be necessary to keep track of matched vertices accessible via odd paths (head vertices), even paths (tail vertices), or both. In notation $X_n = S \cup H_n \cup T_n \cup B_n$. 
\item If there are plenty of edges leaving $X_n$ from $T_n$ or $B_n$, then the other ends of these edges will be part of $X_{n+1}$, fueling the desired growth. The first observation is that if this is not the case, then there has to be many tail-tail or tail-both edges. 
\item A tail vertex that has another tail- or both-type neighbor will normally become a both-type vertex in the next step. In this case even though $X_n$ does not grow, the set $B_n$ grows within $X_n$, still maintaining the desired expansion that eventually leads to a contradiction.
\item The problem is that certain tail-vertices will not become both-type even though they possess a both-type neighbor. These will be called the \textit{tough} vertices. The bulk of the proof is about bounding the number of tough vertices. The key idea here is that we can associate to each tough vertex $x$ a distinct subset of $B_n$ called the \textit{family} of $x$. Families associated to different vertices are pairwise disjoint. (This is done in Section~\ref{combsec}.)
\item There can not be too many tough vertices with large families. On the other hand if a vertex stays tough for an extended amount of time, its family has to grow. These two observation together should be enough to bound the number of tough vertices.
\item The proof proceeds in two rounds from this point. First, if $X_n$ is smaller than half of the graph, then already families larger than $4d(d+1)/c_0$ are too large, and indeed vertices can't be tough too long before they reach this critical family size. Then all the previous observations are valid and $X_n$ grows exponentially as desired. (This is the contents of Theorem~\ref{smallXn_theorem} and the proof is done in Section~\ref{largeoutside_section}.)
\item In the second round, when $X_n$ is already quite big, this unfortunately does not work anymore. The bound after which families can be deemed too large grows as $|X \setminus X_n|$ shrinks, and thus vertices can be tough longer and longer before their families become big enough. At this point it becomes necessary to show that the families of tough vertices also grow exponentially fast. 
\item In Section~\ref{section:familybusiness} we demonstrate that the dynamics of how a family grows is almost identical to how the sets $X_n$ are growing. In fact families are more or less what can be reached from the tough vertex by an alternating path. But a family lives within an infinite countable graph, hence it is never bigger than "half of the graph", so only the first round is needed to show exponential growth. Hence Theorem~\ref{smallXn_theorem} has a double gain. It proves the first round for $X_n$, but at the same time it is used to prove fast family growth in the second round. 
\item Once we have established exponential family growth, an approach very similar to the proof of the first round is used to complete Theorem~\ref{shortalternating_theorem} in Section~\ref{section:secondround}. The proofs of both rounds employ a method of defining an invariant whose growth is controlled. But the hidden motivation behind the invariant is what we have outlined in this sketch: if $X_n$ doesn't grow then $B_n$ grows. If $B_n$ doesn't grow either then there have to be many tough vertices. If there are many tough vertices then they have to be tough for a long time. But then their families have to become too big. Finally there is no space for all these big families. 
\item Unfortunately there is a final twist. When analyzing family growth in Section~\ref{section:familybusiness}, we have to introduce certain forbidden edges in each step, through which alternating paths are not allowed to pass momentarily. Hence, to be able to use Theorem~\ref{smallXn_theorem} in this more general scenario, we need to state it in a rather awkward way. Instead of saying that $X_n$ is just what can be reached by alternating paths of length at most $n$, we need to use a recursive definition of $X_n$ taking into account the forbidden edges in each step. But as it is pointed out in Remark~\ref{remark:forbidden}, if one chooses to have no forbidden edges, $X_n$ just becomes what it was in this sketch.
\end{enumerate}

The proof is organized as follows. In Section~\ref{section:forbidden} we introduce the basic recursive construction of the $X_k$ sets using the notion of forbidden edges. We state the key Theorem~\ref{smallXn_theorem} that on one hand provides the proof of the first round, and on the other hand will be used to show exponential family growth.

Tough vertices and families are introduced in Section~\ref{combsec} together with proofs of their basic properties. Then Theorem~\ref{smallXn_theorem} and the first round is proved in Section~\ref{largeoutside_section}, using the invariant-technique.

In Section~\ref{section:familybusiness} we show how the growth of a family can be modeled using the forbidden edge construction, and prove exponential growth of families. Finally in Section~\ref{section:secondround} we finish the proof of the second round, again using the invariant-technique.

\subsection{Forbidden edges}\label{section:forbidden}

We are going to use the following terminology. All alternating paths will start with an unmatched edge, but may end with either kind of edges. If $p = (p_0, p_1, \dots, p_l)$ is an alternating path of length $|p| = l$, then the vertices with odd index will be referred to as the "head" vertices of $p$ and the even index vertices (except for $p_0$) will be called "tail" vertices. $p$ will be called even if $l$ is even, and odd if $l$ is odd. The last vertex will be denoted by $\last(p) = p_{l}$.
When this doesn't cause confusion, we will also use $p$ to denote just the set of vertices of the path.

\begin{defin}\label{xk_definition}
Assume that for every $k$ we are given a subset of "forbidden" edges $E_k \subset E$. Using this as input data, we shall recursively construct a sequence of vertex sets \[S = X_0 \subset X_1 \subset X_2 \subset \dots .\] Suppose we have already defined $X_k$. Then $X_{k+1}$ is defined as follows. Take a matched edge $vw$ outside of $X_k$. We are going to include these two vertices in $X_{k+1}$ if and only if there is an even alternating path starting in $S$ whose length is at most $2k+2$, whose last two vertices are $v$ and $w$ in some order while all the previous vertices are in $X_k$, and, most importantly, the edge on which it leaves $X_k$ does not belong to $E_k$. 
\end{defin}

\begin{remark}\label{remark:forbidden}
This definition implies that each $X_k$ consists of matched pairs, and for any vertex $v \in X_k$ there is an alternating path $p \subset X_k$ such that $p_0 \in S$, $|p| \leq 2k$, and $\last(p) = v$. If the $E_k$ are all empty, then $X_k$ consists of all vertices accessible from $S$ via an alternating path of length at most $2k$. First we will show that the size of $X_k$ grows fast.
\end{remark}

\begin{theorem}\label{smallXn_theorem} Suppose that
\begin{enumerate}
\item $|X_n| \leq |X \setminus X_n|$, 
\item there are no augmenting paths of length at most $2n-1$ starting in $S$, and 
\item  $|E_k| \leq d|S|$ for all $0 \leq k < n$,
\item the number of non-forbidden edges leaving $X_k$ is at least $1/(d+1)$ portion of all edges leaving $X_k$ for all $k < n$.
\end{enumerate} 
Then 
\[ |X_n| \geq \frac{c_0^2 |S|}{16d^2(d+1)^2} \left(1+\frac{c_0^3}{128 d^3(d+1)^3}\right)^n.\]
\end{theorem}

Note that the first condition is always satisfied in the countable case, since $X_n$ is always finite.

We will need a more refined classification of the vertices in $X_n$. First of all let $\path_o$ denote the set of all odd alternating paths starting from $S$, and $\path_e$ the set of all even alternating paths. For every $n \geq 1$ let us define the following subsets of $X_n$. Let 
\[ \tilde{H}_n = \{ x \in X_n : \exists p \in \path_o (1 \leq |p| \leq 2n, p \subset X_n; \last(p) = x)\},\] 
\[ \tilde{T}_n = \{ x \in X_n : \exists p \in \path_e (2 \leq |p| \leq 2n, p \subset X_n; \last(p) = x)\},\] 
\[  
H_n = \tilde{H}_n \setminus \tilde{T}_n,\] \[ T_n = \tilde{T}_n \setminus \tilde{H}_n, \] \[ B_n = \tilde{H}_n \cap \tilde{T}_n.
\]

It is important that in these definitions we are not insisting that the paths avoid forbidden edges at any time. The forbidden edges only limit the definition of $X_n$, but then we want to consider all possible alternating paths within the set. 

The last three are the set of head vertices, the set of tail vertices, and those that can be both heads or tails. It is clear that $S$ and $T_n$ are disjoint. As long as there are no augmenting paths of length at most $2n-1$, then $S$ is also disjoint from  $\tilde{H}_n$, and thus $X_n$ is a disjoint union of $S, H_n, B_n$, and $T_n$. It follows from the definition that  $B_1 \subset B_2 \subset \dots$, furthermore $M$ gives a perfect matching between $T_n$ and $H_n$, and also within $B_n$. (Note that this implies $|H_n| = |T_n|$.)

The rough idea of why $X_n$ should grow fast is this. By expansion, even in the presence of forbidden edges, there are plenty of edges leaving $X_n$. Any edge leaving $X_n$ from $\tilde{T}_n$ adds to the size of $X_{n+1}$ directly. Only edges leaving from $H_n$ cause problems. But since $H_n$ and $T_n$ have the same total degree, any surplus of edges leaving $H_n$ have to be compensated by edges within $T_n$ or between $B_n$ and $T_n$. Such edges will contribute to the growth of $B_n$ within $X_n$, and thus implicitly to the growth of $X_n$.

\subsection{Combinatorics of alternating paths}\label{combsec}

In this section we will be mainly concerned about how edges within $T_n \cup S$ and between $B_n$ and $T_n \cup S$ contribute to the growth of $B_n$.

\begin{lemma}\label{TTedge_lemma} If $x,y \in T_n \cup S$ and $xy \in E$ then either $x \in B_{n+1}$ or $y \in B_{n+1}$ or there is an augmenting path of length at most $2n+1$.
\end{lemma}

\begin{proof} It is sufficient to prove that either $x$ or $y$ would be in $\tilde{H}_{n+1}$. Let $p$, respectively $q$ be shortest alternating paths that witness $x$ and $y \in T_n$ respectively. We may assume without loss of generality that $|p| \leq |q|$. Then $y$ cannot lie on $p$, otherwise there would either be a shorter alternating path witnessing $y \in T_n$, or we would have $y \in \tilde{H}_n$ and not in $T_n \cup S$. Hence adding the $xy$ edge to $p$ we obtain an alternating path of length at most $2n+1$ that witnesses that $y \in \tilde{H}_{n+1}$.
 \end{proof}

Edges running between $T_n \cup S$ and $B_n$ are more complicated to handle. If $b \in B_n$ and $t \in T_n \cup S$, but all paths witnessing $b \in \tilde{T}_n$ run through $t$, then we can't simply exhibit  $t \in \tilde{H}_{n+1}$ by adding the $bt$ edge to the end of such a path since it would become self-intersecting. The following definition captures this behavior.

\begin{defin}\label{tough_def}~ 
\begin{itemize} \item A vertex $x \in T_n \cup S$ is "tough" if it is adjacent to one or more vertices in $B_n$, but $x \not \in \tilde{H}_{n+1}$. 
\item An edge $xy \in E$ is "tough" if $x \in T_n \cup S, y \in B_n$ and $x$ is a tough vertex.
\end{itemize}
$TT_n$ will denote the set of vertices that are tough at time $n$.
\end{defin}

 We would like to somehow bound the number of tough vertices. In order to do so, we will associate certain subsets of $X_n$ to each tough vertex in a way that subsets belonging to different tough vertices do not intersect. Then we will show that these subsets become large quickly.

\begin{remark} We think of $n$ as some sort of time variable, and all the sets evolve as $n$ changes. Usually $n$ will denote the "current" moment in this process. In the following definitions of age, descendent, and family, there will be a hidden dependence on $n$. When talking about the age or the family of a vertex, we always implicitly understand that it is taken at the current moment. 
\end{remark}

\begin{defin}\label{age_def}
The "age" of a vertex $x \in TT_n$ is $a(x) = n - \min\{k: x \in T_k \cup S\}$.
\end{defin}

\begin{defin}\label{descendent_def} Fix a vertex $x\in TT_n$. A set $D \subset X_n$ has the "descendent property" with respect to $x$ if the following is true. For every $y \in D$ there are two alternating paths $p$ and $q$ starting in $x$ and ending in $y$, such that 
\begin{itemize}
\item both start with an unmatched edge, but $p$ is odd while $q$ is even,
\item $p, q \subset D \cup \{x\}$,
\item $|p| + |q| \leq 2a(x)+1$.
\end{itemize}
\end{defin}

Sets satisfying the descendent property with respect to $x$ are closed under union.

\begin{defin}\label{family_def} The "family" of a vertex $x\in TT_n$ is the largest set $D \subset X_n$ that satisfies the descendent property. In other words it is the union of all sets that satisfy the descendent property. The family of $x$ is denoted by $F_n(x)$.
\end{defin}

\begin{claim}\label{claim:tough_edge} If $x \in TT_n$ and $xy$ is a tough edge then $y$ is in the family of $x$. In particular every tough vertex has a nonempty family. 
\end{claim}

\begin{proof} Let $p$ be a path that witnesses $y \in \tilde{T}_n$. Now if $p$ appended by the edge $yx$ would be a path then it would witness $x \in \tilde{H}_{n+1}$. Since this is not the case,  $x$ has to lie on $p$. Let $D$ denote the set of vertices $p$ visits after leaving $x$. For any point $z \in D$ there are two alternating paths from $x$ to $z$. One is given by $p$ and the other by going from $x$ to $y$ and then walking backwards on $p$. Suppose $x = p_{2l}$ and $y = p_{2k}$. The total length of these two paths is $2k-2l +1$. Since the age of $x$ by Definition~\ref{age_def} is at least $k-l$ we see that $2k-2l+1 \leq 2a(x)+1$. Hence the two paths satisfy all conditions of Definition~\ref{descendent_def} so $D$ has the descendent property with respect to $x$. Hence by Definition~\ref{family_def} $x$ has a non-empty family, in particular $y$ is in the family. 
 \end{proof}

\begin{claim} The family of any tough vertex is a subset of $B_n$. 
\end{claim}

\begin{proof} Let $x \in TT_n$ be a tough vertex and let $s$ be a shortest path witnessing $x \in T_{n} \cup S$. Let us denote $|s| = 2k$. It is enough to show that the family of $x$ is disjoint from $s$. Indeed, then for any point $y$ in the family we can take the two types of paths $p,q$ as in Definition~\ref{descendent_def} from $x$ to $y$. By the age requirement in Definition~\ref{descendent_def} we get that $|p| + |q| \leq 2a(x) + 1 = 2n - 2k +1$.  Hence $|s| + |p| + |q| \leq 2n+1$ and thus $|s|+|p| \leq 2n-1$ and $|s| + |q| \leq 2n$. Since these paths run within the family which is disjoint from $s$, we can append $s$ with ${p}$ and ${q}$ respectively to get alternating paths witnessing $y \in \tilde{H}_n$ and $y \in \tilde{T}_n$ respectively. 

Now suppose the family of $x$ is not disjoint from ${s}$. It is clear that any family consists of pairs of matched vertices. Let $i$ be the smallest index such that the pair $s_{2i-1}, s_{2i}$ is in the family. Then from $x$ there is an odd alternating path ${p}$ to $s_{2i}$ by Definition~\ref{descendent_def}  that runs within the family and its length is at most $2a(x)+1 \leq 2n-2k+1$. Since $i$ was the smallest such index, the path ${p}$ is disjoint from $s_0, s_1, \dots s_{2i-1}$. Thus by appending $s_0, s_1,\dots,s_{2i}$ by the reverse of ${p}$ we get an alternating path from an unmatched point to $x$ ending in an unmatched edge, whose length is at most $2n-2k+1 + 2i \leq 2n+1$. This path witnesses $x \in \tilde{H}_{n+1}$, contradicting the toughness of $x$.  
 \end{proof}

Next we will prove that any vertex can belong to at most one family. We start with a simple lemma about concatenating alternating paths.

\begin{lemma} Let ${p}$ be an even alternating path from $x$ to $y$ and ${q}$ an odd alternating path from $y$ to $z$. Then there is an odd alternating path from $x$ to either $y$ or $z$ whose length is at most $|p| + |q|$. 
\end{lemma}

\begin{proof} If the concatenation of ${p}$ and ${q}$ is a path, then we are done. Otherwise let $i$ be the smallest index such that $p_i \in {q}$. Let $p_i = q_j$. Then $p_0,p_1,\dots, p_i = q_j, q_{j+1}, \dots, \last(q)$ is a path from $x$ to $z$ and $p_0,p_1,\dots,p_i=q_j, q_{j-1},\dots q_0$ is a path from $x$ to $y$. Both have length at most $|p|+|q|$, both of them end with non-matched edges and one of them is clearly alternating. 
 \end{proof}

\begin{claim}\label{disjointfamilies} Two families cannot intersect.
\end{claim}

\begin{proof} Let $x, y \in TT_n$ be two tough vertices. Assume their families $F$ and $G$ do intersect. Let ${p}, {q}$ be shortest alternating paths witnessing $x, y \in T_n \cup S$. Let us choose the shortest among all alternating paths from $x$ to $F \cap G$ that runs within $F$. Let this path be ${p'}$ and its endpoint $x' \in F \cap G$. Do the same thing with $y$ to get a path ${q'}$ from $y$ to $y' \in F \cap G$ lying within $G$. By symmetry we may assume that $|p| + |p'| \leq |q| + |q'|$. 

By the choice of ${p'}$ we see that the only point on ${p'}$ that is in $G$ is its endpoint $x'$. From $x'$ there are two paths, ${s}$ and ${t}$,  leading to $y$ within $G$ by Definition~\ref{descendent_def} one of which, say ${s}$, can be appended to ${p'}$ to get an alternating path from $x$ to $y$. This path ${p'} \cup {s}$ clearly starts and ends with a non-matching edge.

Now we are in a situation to apply the previous lemma. ${p}$ leads from $p_0$ to $x$ and ends with a matching edge, and ${p'} \cup {s}$ leads from $x$ to $y$ and starts and ends with non-matching edges. Thus by the lemma, there is an alternating path from $p_0$ to either $x$ or $y$ which ends with a non-matching edge. The length of this alternating path is at most $|p| + |p'| + |s|$. 
But by the choice of ${p'}$, the choice of ${q'}$, and by the age requirement in Definition~\ref{descendent_def} we have 
\[ |p| + |p'| + |s| \leq |q| + |q'| + |s| \leq |q| + |t| + |s| \leq |q| + 2a(y) + 1 = 2n+1.\]
Thus the alternating path we have found from $p_0$ to $x$ or $y$ has length at most $2n+1$ so it witnesses $x \in \tilde{H}_{n+1}$ or $y \in \tilde{H}_{n+1}$. But neither is possible since both $x$ and $y$ are tough, which is a contradiction.
 \end{proof}

\begin{corol}\label{onlyonetough} There is exactly one tough vertex adjacent to any family.
\end{corol}

\begin{proof} Let $x,y \in TT_n$ and $z \in F_n(x)$. Suppose there is an edge between $y$ and $z$.Then $z$ is in $B_n$, hence $yz$ is a tough edge, hence $z$ is in the family of $y$, but then the two families would not be disjoint, which is a contradiction.
 \end{proof}

Let $c_1 = c_1(c_0,d)$ be a constant to be determined later.

\begin{claim}\label{expanding_claim} Suppose $|F_n(x)| < c_1$, $v \in F_n(x)$, and there is an edge $vw$ such that $w \in B_n \setminus F_n(x)$. Then either $w \in F_{n+c_1}(x)$ or $x \in \tilde{H}_{n+c_1}$. In other words, if a vertex remains tough for an extended period of time, then its family consumes its neighbors. 
\end{claim}

\begin{proof} We can assume that $x \not \in \tilde{H}_{n+c_1}$ since otherwise we are done. Thus $x$ is still tough at the moment $n+c_1$.

First suppose there is a path $p \in \path_e, |p| \leq 2n$ that ends in $w$ and does not pass through $x$. Let $w' \in p$ be the first even vertex on the path that is adjacent to some vertex $v' \in F_n(x)$. Then the initial segment of $p$ up until $w'$ has to be disjoint from $F_n(x)$. By definition, in $F_n(x)$ there has to be an alternating path from $x$ to $v'$ that ends in a matched edge. Extending this path through $w'$ and then the initial segment of $p$, we get an alternating path from $S$ to $x$. Its length is obviously at most $|p| + c_1$, hence $x \in \tilde{H}_{n+c_1/2}$ and consequently in $\tilde{H}_{n+c_1}$, and this is a contradiction.

 That means that any even path from $S$ to $w$ of length at most  $2n$ has to pass through $x$. Let $p$ be the shortest such path. Let $v'$ be the last vertex of $p$ that is in $F_n(x) \cup \{x\}$. The vertex $v'$ divides $p$ into two segments, $p_1$ going from $S$ to $v'$ and $p_2$ from $v'$
to $w$. Then $|p_2| = |p| - |p_1| \le 2n - 2 \min \{ k : x \in T_k \cup S\}
= 2a(x)$, and equality can only happen if $x = v'$. We claim that $p_2$ becomes part of the family at time $n+c_1$.
For any vertex $y \in p_2$ we can either go from $x$ to $v'$ in even steps and then continue along $p_2$, or go from $x$ to $v$ in even steps and continue backwards on $p_2$ to $y$. The total length of these two paths is at most  $c_1 + |p_2| + 1 + c_1 \leq 2(a(x)+c_1) + 1$.  Since at moment $n + c_1$ the age of $x$ is exactly $a(x)+c_1$, the set $F_n(x) \cup p_2$ will satisfy the descendent property, so this whole set, including $w$, will be part of $F_{n+c_1}(x)$.
\end{proof}

\begin{defin}\label{expanding_def} We will say that at moment $n$ the family of the vertex $x \in TT_n$ is \textit{expanding} if  there is an edge $vw$ such that $v \in F_n(x)$ and $w \in B_n \setminus F_n(x)$.
For any $x \in X$, let $e_n(x)$ be the number of moments $m < n$ such that $0 < |F_m(x)| < c_1$ and at moment $m$ the family was expanding. 
\end{defin}

\begin{corol}\label{bounded_en_corol} For any $x\in X$ we have $e_n(x) \leq c_1^2$ independently of $n$. 
\end{corol}

\begin{proof} By Claim~\ref{expanding_claim} we know that the number of moments in which an expanding family has a fixed size $k < c_1$ is at most $c_1$. This is because after the first such moment, in $c_1$ time the family either ceases to exist or strictly grows. Thus for each possible size $k$ there are at most $c_1$ moments of expansion, and thus there are at most $c_1^2$ such moments in all.
\end{proof}

\subsection{Invariants of growth}\label{largeoutside_section}

Now we are ready to start the proof of Theorem~\ref{smallXn_theorem}. Let 
\[
I(n) = |X_n| + |B_n| + \frac{1}{2}\int_X e_n(x) dx.
\]
or in the infinite connected case
\[
I(n) = |X_n| + |B_n| + \frac{1}{2}\sum_{x\in X} e_n(x).
\]

\begin{propo}\label{smallXn_propo} Suppose that 
\begin{enumerate}
\item $|X_n| \leq |X \setminus X_n|$, 
\item there are no augmenting paths of length at most $2n-1$ in $X_n$, and 
\item  the number of forbidden edges is $|E_k| \leq d|S|$ for all $0 \leq k < n$,
\item the number of non-forbidden edges leaving $X_k$ is at least $1/(d+1)$ portion of all edges leaving $X_k$ for all $k < n$. 
\end{enumerate} 
 then 
 \[ I(n+1) \geq \left(1+\frac{c_0^3} {128d^3(d+1)^3}\right)I(n).\]
\end{propo}

\begin{proof}
In the following we shall omit the index $n$ from all our notation, except where it would lead to confusion. Let $TT$ denote the set of tough and $TM$ the set of not-tough vertices within $T \cup S$. The tough vertices are further classified according to their families. $TB$ denotes the tough vertices whose families have size at least $c_1$.  For tough vertices with smaller families, $TE$ shall denote the ones that have expanding families and $TG$ denote the rest. So 
\[ S \cup T = TM \cup TT = TM \cup \left(   TB \cup TE \cup TG \right).\]

First let's take a tough vertex $x \in TG$ whose family is small and not expanding. Let $|E(x, F(x))| = k$. By the assumption on the size of the smallest real odd cut we know that the number of edges leaving $x \cup F(x)$ (which is a set of odd size!) is at least $d+1$. But only $d-k$ of these are adjacent to $x$, so at least $k+1$ have to be adjacent to $F(x)$. None of these edges can lead to $B$ because this is a non-expanding family. Also non of these edges can lead to $TT$  by Corollary~\ref{onlyonetough}. Hence all these edges have to go to $H$, $TM$, or the outside world $O = X^c$. This means that 
\begin{equation}\label{eq:xFx} |E(F(x), TG)| = |E(F(x),x)| \leq |E(F(x), H \cup TM \cup O)|.\end{equation}
By Claim~\ref{claim:tough_edge} we see that any edge between $TG$ and $B$ has to run between a vertex in $TG$ and a member of its family. Thus integrating~(\ref{eq:xFx}) over $x\in TG$ and using that families are pairwise disjoint subsets in $B$ we get that 
\[ |E(B, TG)| \leq |E(B, H \cup TM \cup O)| \]

For any other tough vertex we bound the number of edges between it and $B$ by the trivial bound $d$. Adding this to the previous equation we get
\begin{equation} |E(B, TT)|  \leq d|TB| + d |TE| + |E(B, H \cup TM  \cup O)|
  \label{TT_eq} \end{equation}
  
We know that $|T| = |H|$ because of the matching, so the total degrees of $S \cup T$ is $d|S|$ more than the total degree of $H$. The edges between $T \cup S$ and $H$ contribute equally to these total degrees. In the worst case there are no internal edges in $H$. This boils down to the following estimate.
\begin{multline*}
 |E(H, O)| + |E(H, B)| + d|S| \leq \\ \leq 2|E(T\cup S,T\cup S)| + |E(T \cup S,O)| + |E(TM, B)| + |E(TT, B)|.
\end{multline*}
Combining it with (\ref{TT_eq}), and subtracting $|E(H,B)|$ from both sides we get
\begin{multline*} 
|E(H,O)| + d|S| \leq 2|E(T\cup S,T\cup S)| + 2|E(B, TM)| +\\+ |E(B \cup T \cup S, O)| + d|TB| + d|TE|. \end{multline*}
Any vertex in $TB$ has a family of size at least $c_1$, and all these are disjoint by Claim~\ref{disjointfamilies} and contained in $B$. Thus we get that $|TB| \leq |B|/c_1$. Using this and adding $|E(B\cup T \cup S,O)|$ to both sides implies
\begin{multline}\label{XO_eq} 
|E(X_n,O)| + d|S| \leq 2|E(T \cup S,T\cup S) + 2|E(B,TM)| + \\ + 2|E(B\cup T \cup S,O)| + \frac{d}{c_1} |B| + d|TE|.
\end{multline}
Any vertex in $O_n$ that is adjacent to $B_n \cup T_n \cup S$ along an edge not in the forbidden set $E_n$  is going to be in $X_{n+1}$, hence 
\[ |E(B_n \cup T_n \cup S, O_n) \setminus E_n| \leq d(|X_{n+1}| - |X_n|).\]
By definition, any vertex in $TM_n$ that is adjacent to an edge coming from $B_n$ will be part of $B_{n+1}$ or yield an augmenting path. Also, by Lemma~\ref{TTedge_lemma}, any edge in $E(S \cup T_n,S\cup T_n)$ has to be adjacent to a point in $|B_{n+1}| \setminus |B_n|$ or yield an augmenting path. This implies that 
\[ 2|E(T,T)| + 2|E(B,TM)|  \leq 2d(|B_{n+1}| - |B_n|).\]
By the 3rd assumption of the proposition we have $|E_n| \leq d|S|$. Plugging all this into (\ref{XO_eq}) we get
\begin{equation} \frac{|E(X_n,O_n) \setminus E_n|}{d} \leq 2(|X_{n+1}| - |X_n|) + 2(|B_{n+1}| - |B_n|) + |TE| + \frac{|B_n|}{c_1} \label{XO2_eq} \end{equation}
By Definition~\ref{expanding_def}, for any vertex $x \in TE_n$ we get $e_{n+1}(x) = e_n(x) + 1$, and thus 
\[ \int_X e_{n+1}(x)dx = \int_X e_n(x)dx + |TE|.\]
Hence the right hand side of (\ref{XO2_eq}) is exactly $2(I(n+1)-I(n)) + |B_n|/c_1$. Furthermore by the 4th and 1st assumptions of the proposition we have  
\[|E(X_n,O_n) \setminus E_n| \geq  \frac{|E(X_n, O_n)|}{d+1} \geq \frac{c_0 |X_n|(1-|X_n|)}{d+1} \geq \frac{c_0}{2d+2} |X_n|\] in the measurable case and 
\[|E(X_n,O_n) \setminus E_n| \geq  \frac{|E(X_n, O_n)|}{d+1} \geq \frac{c_0 |X_n|}{d+1} \geq \frac{c_0}{2d+2} |X_n|\] 
in the connected infinite case. So in either case we get 
\[
\frac{c_0}{4d(d+1)} |X_n| - \frac{|B_n|}{2c_1} \leq I(n+1) - I(n)
\]

Now we can complete the proof of the proposition. First, choose $c_1 = 4d(d+1)/c_0$. Then, since $|B_n| \leq |X_n$ we get 
\[\frac{|X_n|}{2c_1} \leq I(n+1) - I(n).\]
On the other hand, we know from Corollary~\ref{bounded_en_corol} that $e_n(x) \leq c_1^2$. Obviously $e_n(x) = 0$ if $x \in O_n$. Thus $\int_X e_n(x) dx \leq c_1^2|X_n|$. Hence 
\[ I(n) \leq \left(2+ \frac{c_1^2}{2}\right)|X_n| \leq  c_1^2 |X_n| \leq 2c_1^3 (I(n+1) - I(n)). \]
Substituting $c_1 =   4d(d+1)/c_0$ finishes the proof.
 \end{proof}

This proposition implies that $I(n)$ grows exponentially fast. But as we have seen, $|X_n|$ can be bounded from below in terms of $I(n)$. This will imply fast growth of $|X_n|$ too.

\begin{proof}[Proof of Theorem~\ref{smallXn_theorem}]~

Since $S \subset X_0$, we have $|S| \leq I(0)$. Then again by Corollary~\ref{bounded_en_corol} we have $I(n) \leq c_1^2|X_n|$. So by Proposition~\ref{smallXn_propo} 
\[ X(n) \geq \frac{I(n)}{c_1^2} \geq \frac{|S|}{c_1^2} \left(1+\frac{c_0^3} {128d^3(d+1)^3}\right)^n \]
Substituting $c_1 = 4d(d+1)/c_0$ we get the desired result. 
\end{proof}

In the measurable case, when $X_n$ becomes large, the method apparently breaks down. The main problem is that expansion guarantees only  $c_0|X_n| (1-|X_n|)$ edges between $X_n$ and $O_n$. When $X_n$ is large, the $1-|X_n|$ term will be the dominant. It was crucial to choose $c_1$ so that the $|B_n|/c_1$ terms becomes comparable to the lower bound coming from expansion. But for large $B_n$, hence small $1-|X_n|$,  this cannot be done with a constant $c_1$. The smallest $c_1$ that has a chance to work is roughly on the scale of $1 / \ep$. But then the upper bound for $I(n)$ becomes $(1/\ep)^3$ and all of a sudden the time needed for $X_n$ to exceed $1-\ep$ becomes super-linear in $1/\ep$  instead of the desired poly-logarithmic dependence.

This loss of time comes from the part where we argued that any family grows bigger than $c_1$ in $c_1^2$ time. This observation was sufficient for a constant $c_1$, but is clearly insufficient when $c_1 \approx 1/\ep$. In this part we will show that, in fact, families grow much faster than what Claim~\ref{expanding_claim} asserts.  It turns out that in a sense families grow exponentially, hence it takes much less time than $(1/\ep)^2$ to reach a size of $1/\ep$. This will allow us to "fix" the argument in Section~\ref{largeoutside_section}.

\subsection{Family business}\label{section:familybusiness}

In this section we shall examine in detail the lifecycle of a family.  Let us fix a vertex $x \in X$. At some $n_0$, this $x$ may become an element of $T_{n_0}$. Then later it may start to have neighbors in $B_{n_1}$ (for a larger value $n_1 \geq n_0$). At this point it can become tough and start to have a family. This family grows in time, until at some even larger value of $n$ the vertex finally becomes part of $B_n$. We want to understand the part when $x$ becomes tough and its family starts growing. 

To this end we shall recursively define a sequence of "special moments" 
\[ n_0 \leq n_1 < n_2 <  n_3\dots \] and an increasing sequence of sets 
\[\emptyset = FX_0 \subset FX_1 \subset FX_2 \subset FX_3 \dots \] that control how fast the family grows. The definition is rather complicated, so we present it step-by-step, along with the notation. For any $n_i \leq n < n_{i+1}$ we write $c(n) = i$, and think of it as a counter. The sets $FX_k$ are going to be defined such that the following hold:
\begin{enumerate}
\item $FX_k$ is the union of some matched pairs of vertices.
\item For any matched edge $vw \subset FX_k$ there is an alternating path $p$ that starts in $x$, lies entirely in $FX_k$, ends with the matched edge (in either direction) and has length at most $2k$. 
\item $FX_{c(n)} \subset F_n(x)$ holds for all $n$ when $x$ is tough, as shown on this scheme of evolution:
\[ n_0 \xrightarrow{ FX_0 = \emptyset} n_1 \xrightarrow{FX_1 \subset F_{n_1}(x)} n_2 \xrightarrow{FX_2 \subset F_{n_2}(x)} n_3 \xrightarrow{FX_3 \subset F_{n_3}(x)} n_4 \dots \]
\end{enumerate}

Suppose we have already fixed $n_k$ and $FX_k$. 

\begin{defin}\label{def:mk}\label{fxk_table} Let $m_k$ denote the smallest moment $m_k > n_k$ in which there are at most $d$ edges leaving  $FX_k \cup \{x\}$ that do not end in $B_{m_k} \setminus FX_k$. Let this set of edges be denoted by $E_k$. Now define $FX_{k+1}$ to be the extension of $FX_k$ by those matched edges in $B_{m_k}$ that can be the last edge of an alternating path of length at most $2k$ starting from $x$, and lying entirely in $FX_k$ except for its last two vertices. 
\begin{multline} FX_{k+1} = FX_k \cup \{v \in B_{m_k} : \exists p \in \path_e (|p| \leq 2k+2; p_0 = x; \\ p_1,\dots,p_{2k} \in FX_k; p_{2k+1} = v \mbox{ or } p_{2k+2} = v\} \end{multline}
\end{defin}

It is clear that this construction satisfies the first two conditions stated just above Definition~\ref{fxk_table}, but there is no reason for $FX_{k+1}$ to be a subset of $F_{m_k}(x)$.  However, if we choose $n_{k+1} = m_k + 2k$ then the following claim implies that the third condition will be also satisfied.

\begin{claim}\label{becomesfamily_claim}
While $x$ is tough, $FX_{k+1} \subset F_{m_k + 2k}(x)$ for all $k$,  hence $FX_{c(n)} \subset F_n(x)$ for all $n$.
\end{claim}

\begin{proof} This is very similar to Claim~\ref{expanding_claim}. We argue by induction on $k$. Then we can assume that $FX_k \subset F_{m_k}$. We need to show that $FX_{k+1} \subset F_{m_k+2k+1}$. Take a matched edge $vw \subset FX_{k+1} \setminus FX_k$. By definition there is an alternating path $p$ of length at most $2k+2$ starting in $x$, ending in the $vw$ edge, and lying in $FX_k$.  Suppose its last vertex is $w$. Since $vw \subset B_{m_k}$, there has to be a path $q$ proving this, ending in the same edge, but in the opposite order: $wv$. Let's take the shortest such path. It has to pass through $x$, otherwise $x$ would not be tough at $n = m_k+k$. Denote the part of this path between $x$ and $v$ by $q$. Now we have two paths from $x$. The path $p$ ends with $vw$ while the path $q$ ends with $wv$. The length of $p$ is at most $2k+2$, the length of $q$ is at most $2a(x)$. We will show that some subset of $q$ together with $F_{m_k}$ satisfies the descendent property at $n = m_k+2k+1$. 

\begin{lemma}\label{intertwinedpaths_lemma} Suppose $p$ and $q$ are alternating paths, both starting with a non-matched edge from the same vertex $x$ and ending in a matched edge $vw$ but from different directions.  Then there is a subset $U \subset q$ containing $v$ and $w$, such that for each vertex $z \in U$ there are two alternating paths between $x$ and $z$ of different length-parities, lying entirely in $U \cup p$, whose total length is at most $|q| + 2|p| - 3$.
\end{lemma}

Before giving the proof of the lemma, let us show how this completes the proof of the claim. It is easy to see that $U \cup F_{m_k}$ satisfies the descendent property at time $m_k+2k$. First of all, by definition, the set $F_{m_k}$ itself satisfies it. On the other hand for any vertex in $U$ the lemma guarantees the existence of the two alternating paths lying entirely in $U \cup p \subset U \cup F_{m_k}$, since $p \subset F_{m_k}$ by induction. The sum of the length of these two paths  is at most $2a(x) + 2(2k+2) - 3 = 2a(x) + 4k +1$. The age of $x$ at $n = m_k+2k$ is $a(x) + 2k$ and so we are done. This completes the proof of the induction step, hence the claim is true.
\end{proof}

\begin{proof}[Proof of Lemma~\ref{intertwinedpaths_lemma}]
If $p$ and $q$ are disjoint apart from their endpoints, then the statement is obvious with $U = q$, and we even get the stronger upper bound $|q|+|p|$ on the total length of the two paths for any vertex in $U$. If $p$ and $q$ are badly intertwined, we need to be cautious. 
Let $x=q_0, q_1,\dots, q_{2l} = v$ denote the vertices of $q$.  Since both $p$ and $q$ are alternating paths, their intersection is necessarily a union of matched edges. For each matched edge $q_{2i-1} q_{2i}$ the path $p$ may contain this edge, or not. The ones that are contained in $p$ will be called double edges. For each double edge,  $p$ may contain it in the same orientation as $q$ - these will be called good double edges,  or the opposite orientation as $q$ - these will be called bad double edges. 

There are two natural partial orders on the set of matched edges of $p$ and $q$. For two such edges $e$ and $f$ will write $e <_q f$ if $e$ comes before $f$ on the path $q$. We will write $e <_p f$ if $e$ comes before $f$ on $p$. (If one or both of the edges aren't on a given path, they are incomparable in the given order.) Now for any matched edge $e$ on $q$, we define 
\[ Z(e) = \min_{<_p} \{ f : f \geq_q e\}.\]
Note that, since the $q$-maximal edge $vw$ is a double edge, $Z(e)$ is always well-defined. Also note that $Z(Z(e)) = Z(e)$. Next, let 
\[ f = \max_{<_q} \{ e \in q : Z(e) = e \mbox{ is a good double edge}\},\]
and let $x'$ be the vertex of $f$ further away from $x$. If there is no such double edge, then $f$ is not defined, and we just choose $x' = x$. Let $q'$ be the part of $q$ from $x'$ to $v$, let $p'$ be the part of $p$ between $x'$ and $w$, and let $p''$ be the part of $p$ between $x$ and $x'$.
We claim that $U = q' \setminus p$ is a good candidate. 

First of all, observe that $p'' \cap q' = x'$. When $x' = x$ this is obvious. Otherwise it is still true because $Z(f) = f$, which means that any edge in $q'$ is visited by $p$ later than $f$ is visited by $p$. Second, take any matched edge $e \in q'$.  By definition, $f <_q e$. Hence 
\[ f <_q e \geq_q Z(e) = Z(Z(e)),\] so by construction $Z(e)$ has to be a bad double edge. Now we can exhibit the two alternating paths between $x$ and the edge $e$.

 From one direction we can simply reach it by going on $p''$ until  $x'$ and then continuing on $q'$ until we reach $e$. This is a path, since $p'' \cap q' = x'$.  From the other direction, start at $x$ and go on $p''$ to $x'$ and then further on $p'$ until hitting $Z(e)$. Since $Z(e)$ is a bad double edge, we have just visited it in the 'wrong' direction on $q$. So we can now continue on $q$ backwards from $Z(e)$ until we come to $e$. The concatenation of these two segments is still a path, since by definition of $Z(e)$, the part of $p$ between $x$ and $Z(e)$ is disjoint from the part of $q$ between $e$ and $Z(e)$. The total length of the two paths we have just exhibited is at most $2|p''| + |q'| + |p'| -1$. The $-1$ comes from the fact that the $vw$ edge is contained in both $p$ and $q$, but has to be used at most once. Finally $|p'| \geq 2$ thus the total length is at most $|q| + 2|p| -3$.
\end{proof}

Now let's look at the connected component of $x$ denoted by $X'$. It's a (finite or countable) connected $d$-regular $c_0$-expander graph with a partial matching. Let's remove the edge containing $x$ from the matching. Let $S' = \{x\}$ and let $X'_k = \{x\} \cup FX_k$. We have already defined the sets $E_k$ that contain all the edges leaving $X'_k$ not ending in $B_{m_k}$, hence in particular containing the once matched edge coming out of $x$. The sets $X'_k$ were constructed exactly according to the rules of Definition~\ref{xk_definition}. Clearly $|E_k| \leq d|S'|$. But since any odd set, in particular $X'_k$, has at least $d+1$ edges leaving it, of which at most $d$ is forbidden, the 4th assumption of Theorem~\ref{smallXn_theorem} 
is also satisfied. Thus it applies in this situation and implies that as long as $x$ remains tough and $|F_n(x)| \leq |X' \setminus F_n(x)|$, we have 
\[ |F_n(x)| \geq |FX_{c(n)}| \geq \frac{c_0^2 |S'|}{16d^4} \left(1+\frac{c_0^3}{128 d^6}\right)^{c(n)}.\]
In the countable case the $|F_n(x)| \leq |X' \setminus F_n(x)|$ condition is always satisfied and $|S'| = 1$, while in the finite case it is satisfied as long as the family doesn't occupy at least half of the graph, and $|S'| = 1/|X|$. Thus in both cases we get 

\begin{corol}\label{largefamily_corol} As long as $x$ remains tough and $||F_n(x)|| < |X| / 2$,
\[ ||F_n(x)|| \geq ||FX_{c(n)}|| \geq \frac{c_0^2}{16d^4} \left(1+\frac{c_0^3}{128 d^6}\right)^{c(n)},\] where $|| \cdot ||$ denotes the actual size of the set in both the finite and the countable cases.
\end{corol}

\begin{defin}\label{def:dormant} Let us say, that for such moments when $n_i \leq n < m_i$ for some $i$, the family of $x \in X$ is \textit{dormant}, whereas for moments that satisfy $m_i \leq n < n_{i+1}$ the family is \textit{active}. Let $f_n(x)$ denote the number of moments $m < n$ such that $||F_m(x)|| < c_3/\ep$ and at moment $m$ the family was active, where $\ep$ denotes the ratio of unmatched vertices in $X$. We will choose $c_3 =2$ except in the case when $X$ is finite and $\ep = 2/||X||$. In this case we will choose $c_3 = 1$.
\end{defin}

It is clear that $f_n(x) \leq c(n)^2$. Note that in the finite case either $\ep \geq 4/||X||$ and thus $c_3 = 2$ and $c_3/\ep \leq ||X||/2$, or $\ep = 2/||X||$ and $c_3/\ep = ||X||/2$. Hence families that haven't reached the size $c_3/\ep$ are not bigger than half of the graph. Hence by~\ref{largefamily_corol} we have in both the measurable and the finite case that
\begin{equation}\label{fnbound_eq}
 f_n(x) \leq \left(\frac{\log \left (\frac{8c_3d^4}{\ep c_0^2}\right)}{ \log \left(1+ \frac{c_0^3}{128d^6}\right)}\right)^2 
\leq c_4 \left(1+\log^2 1/\ep\right)
\end{equation}
for a suitably large $c_4$ depending only on the previous constants and $d$.

\subsection{Proof of Theorem~\ref{shortalternating_theorem}}\label{section:secondround}

The proof will work similarly to that of Theorem~\ref{smallXn_theorem}, but one has to be more careful. This time we are interested only in the measurable case, and assume that all the sets $E_k$ of forbidden edges are empty. Thus $\tilde{H}_k$ is simply the set of vertices that can be the end-point of an odd alternating path of length at most $2k-1$ starting in $S$. We choose $S$ to be half of set of unmatched vertices. Then as soon as we have $S \cap \tilde{H}_k \not = \emptyset$ or $F \cap X_n \not = \emptyset$, we have found an augmenting path.

Let 
\[ J(n) = |X_n| + |B_n| + \frac{1}{2}\int_X f_n(x)dx.\] 
We further reintroduce the notation from the proof of Theorem~\ref{smallXn_theorem}. As before, we will often drop the index $n$, when it does not cause confusion.
 Let $TT$ denote the set of tough and $TM$ the set of not-tough vertices within $T \cup S$. The tough vertices are further classified according to their families. $TB$ denotes the tough vertices whose families have size at least $c_3/\ep$.  For tough vertices with smaller families, $TE$ shall denote the ones that have active families at the moment, and $TG$ denote the ones that have dormant families at the moment. So 
\[ S \cup T = TM \cup TT = TM \cup \left(   TB \cup TE \cup TG \right).\]

First let's take a tough vertex $x \in TG$ whose family is small and dormant. By Definition~\ref{def:mk} this means, that there are at least $d+1$ edges leaving $x \cup FX_{c(n)}$ that do not end in $B_n$. Let $|E(x, FX_{c(n)})| = k \leq d$. Then there are $d-k$ edges leaving $x \cup FX_{c(n)}$ from $x$. The rest, at least $k+1$ must leave from $FX_{c(n)}$. And since these edges do not end in $B_n$, they actually have to leave the whole family $F_n(x)$. The only tough vertex adjacent to the family is $x$ by Corollary~\ref{onlyonetough}, so the $k+1$ edges we have just exhibited must end in $H \cup TM \cup O$. When $k \leq d$, then $(k+1)d/(d+1) \geq k$. So we have 
\[|E(F(x), TG)| = |E(F(x),x)| \leq \frac{d}{d+1}|E(F(x), H \cup TM \cup O)|.\] 
Integrating over $TG$ we get that 
\[ |E(B, TG)| \leq \frac{d}{d+1}|E(B, H \cup TM \cup O)| \]
For any other tough vertex we bound the number of edges between it and $B$ by the trivial bound $d$. Adding this to the previous equation we get
\begin{equation} |E(B, TT)|  \leq d|TB| + d |TE| + \frac{d}{d+1}|E(B, H \cup TM  \cup O)|
  \label{TT_new_eq} \end{equation}
Now let us examine the edges running between $B_n$ and its complement. By~(\ref{TT_new_eq}) we have 
\begin{multline*} |E(B, X \setminus B) =  |E(B,H\cup O \cup TM)| + |E(B,TT)| \leq  \\ \leq
2|E(B,H \cup O \cup TM)| + d|TB|+ d|TE| 
\end{multline*}  and hence
\[ \frac{|E(B,X \setminus B)|}{2(d+1)} \leq \frac{|E(B,H \cup TM \cup O)|}{d+1} + \frac{1}{2}(|TB|+|TE|).\]
Adding this to (\ref{TT_new_eq}) then yields 
\begin{equation}\label{TTB_new_eq}
\frac{|E(B,X \setminus B)|}{2(d+1)} + |E(B,TT)| \leq |E(B,H \cup TM \cup O)| + (d+1)(|TB|+|TE|).
\end{equation}

We know that $|T| = |H|$ because of the matching, so the total degrees of $S \cup T$ is $d|S|$ more than the total degree of $H$. The edges between $T \cup S$ and $H$ contribute equally to these total degrees. In the worst case there are no internal edges in $H$. This boils down to the following estimate.
\begin{multline*}
 |E(H, O)| + |E(H, B)| + d|S| \leq \\ \leq 2|E(T\cup S,T\cup S)| + |E(T \cup S,O)| + |E(TM, B)| + |E(TT, B)|.
\end{multline*}
Adding $\frac{|E(B,X \setminus B)|}{2(d+1)}$ to both sides, then using~(\ref{TTB_new_eq}), and subtracting $|E(H,B)|$ from both sides we get
\begin{multline*} 
|E(H,O)| + \frac{|E(B,X \setminus B)|}{2(d+1)}+ d|S| \leq 2|E(T\cup S,T\cup S)| + 2|E(B, TM)| +\\+ |E(B \cup T \cup S, O)| + (d+1)(|TB| + |TE|). \end{multline*}
Adding $|E(B\cup T \cup S,O)|$ to both sides implies
\begin{multline}\label{XO_new_eq} 
|E(X_n,O)| + \frac{|E(B,X \setminus B)|}{2(d+1)} + d|S| \leq 2|E(T \cup S,T\cup S) + 2|E(B,TM)| + \\ + 2|E(B\cup T \cup S,O)| + (d+1)\left(|TB| + |TE|\right).
\end{multline}

Any vertex in $O_n$ that is adjacent to $B_n \cup T_n \cup S$ is going to be in $X_{n+1}$, hence 
\[ |E(B_n \cup T_n \cup S, O_n)| \leq d(|X_{n+1}| - |X_n|).\]
By definition, any vertex in $TM_n$ that is adjacent to an edge coming from $B_n$ will be part of $B_{n+1}$ or yield an augmenting path. Also, by Lemma~\ref{TTedge_lemma}, any edge in $E(S \cup T_n,S\cup T_n)$ has to be adjacent to a point in $|B_{n+1}| \setminus |B_n|$ or yield an augmenting path. This implies that 
\[ 2|E(T,T)| + 2|E(B,TM)|  \leq 2d(|B_{n+1}| - |B_n|).\]
 Plugging all this into (\ref{XO_new_eq}) we get
\begin{multline} \frac{|E(X_n,O_n)|}{d+1} + \frac{|E(B,X \setminus B)|}{2(d+1)^2} + |S| \leq  \\ \leq \frac{2d}{d+1}\left(|X_{n+1}| - |X_n| + |B_{n+1}| - |B_n|\right) + |TE| + |TB| \label{XO2_new_eq} \end{multline}
By Definition~\ref{expanding_def}, for any vertex $x \in TE_n$ we get $f_{n+1}(x) = f_n(x) + 1$, and thus 
\[ \int_X f_{n+1}(x)dx = \int_X f_n(x)dx + |TE|.\]
Hence the right hand side of (\ref{XO2_new_eq}) is at most  $2(J(n+1)-J(n)) +  |TB|$. Furthermore by the expander assumption we have  
\[|E(X_n,O_n)| \geq  c_0 |X_n|(1-|X_n|)\] and 
\[|E(B_n, X \setminus B_n)| \geq c_0 |B_n|(1-|B_n|)\] so from~(\ref{XO2_new_eq}) we get 
\begin{equation}\label{eq:stb}
\frac{c_0|X_n|(1-|X_n|)}{d+1} + \frac{c_0 |B_n|(1-|B_n|)}{2(d+1)^2} + |S| - |TB| \leq  2(J(n+1) - J(n))
\end{equation}

Any vertex in $TB$ has a family of size at least $c_3/\ep$, and all these are disjoint by Claim~\ref{disjointfamilies} and contained in $B$. Thus we get that $|TB| \leq \ep|B|/c_3 \leq \ep/2 = |S|$ in the measurable case and in the finite case when $\ep \geq 4/||X||$. In the finite case when $\ep = 2/||X||$, then any tough vertex in $TB$ has a family of size at least  $||X||/2$, and thus there can be at most one tough vertex. We get $|TB| \leq |S|$ in all cases, and thus 
 \begin{equation}\label{jngrowth_eq}
\frac{c_0|X_n|(1-|X_n|)}{2(d+1)} + \frac{c_0 |B_n|(1-|B_n|)}{4(d+1)^2}  \leq  J(n+1) - J(n).
\end{equation}

If we could prove a similar growth estimate on the size of $X_n$ (or $B_n$), then the next lemma would imply that $X_n$ (or $B_n$) would grow too large in a sufficiently small number of steps, proving the existence of a short augmenting path.

\begin{lemma}\label{expseq_lemma} Let $0< a_0 < a_1 < a_2, \dots$ be an increasing sequence of numbers. Let us fix a constant $c$ and say that an index $k$ is good if $a_{k+1} - a_k \geq 2c a_k (1-a_k)$ holds. Then if the number of good indices up to $N$ is at least  
\[ 2\left\lceil \frac{\log(\frac{1}{2a_0})}{\log(\frac{1}{1- c})}\right\rceil,\] then $a_N > 1 - a_0$.
\end{lemma}

\begin{proof}
Let us split the sequence into two parts. The first part will be where $a_k < 1/2$ and the second part where $a_k \geq 1/2$.

In the first part if $k$ is a good index then $a_{k+1} \geq a_k(1+c)$. Hence $a_k \geq a_0 (1+c)^{g(k)}$ where $g(k)$ denotes the number of good indices up to $k$. So if  
\[ g(k_1) \geq \left\lceil \frac{\log(\frac{1}{2a_0})}{\log(1+c)}\right\rceil \] we must have $a_{k_1} > 1/2$, or in other words $k_1$ already has to be in the second part.

In the second part a good index $k$ implies $1- a_{k+1} \leq (1-a_k)(1-c)$, hence if $N$ is such that 
\[ g(N) = g(k_1) + \left\lceil \frac{\log(\frac{1}{2a_0})}{\log(\frac{1}{1- c})}\right\rceil \leq  2 \left\lceil \frac{\log(\frac{1}{2a_0})}{\log(\frac{1}{1- c})}\right\rceil\] then we 
must have $1-a_{N} < a_0$. 
\end{proof}

The problem is that~(\ref{jngrowth_eq}) doesn't directly imply such a growth estimate on either $X_n$ or $B_n$ because a priori the integral term in $J_n$ could absorb any growth implied by the inequality. We need one final trick to overcome this difficulty. The idea is that we don't need $X_n$ or $B_n$ to grow the desired amount in one single step. If we can find a not so large $K$ such that $|X_{n+K} - X_n| \geq 2c(|X_n|)(1-|X_n|)$, or $|B_{n+K} - B_n| \geq 2c(|B_n|)(1-|B_n|)$, we are still good.  So let us fix some $K$, whose precise value is to be determined later, and assume that 
\[|X_{n+K}| - |X_n| < \frac{c_0}{2}(|X_n|)(1-|X_n|) \mbox{ and } |B_{n+K}| - |B_n| < \frac{c_0}{2}(|B_n|)(1-|B_n|).\]
This means that the growth of $J(n)$ implied by~(\ref{jngrowth_eq}) has to largely come from the $\int f_n$ term. But note that once a vertex $x$ has a positive $f$-value, then it has to be tough for the rest of its life, until it becomes part of $B_m$ for some later $m$, and from that point on its $f$-value remains constant. Hence if for some $x$ we find that $f_{n+K}(x) > f_n(x)$, then either $x \in B_{n+K} \setminus B_n$, or $x \in TT_{n+K}$. Also by~(\ref{fnbound_eq}) we know that $f_{n+K}(x) - f_n(x) \leq c_4 (1+\log^2 1/\ep)$. Hence we get
\begin{equation}\label{Kstepbound_eq} \int_X f_{n+K}(x) dx - \int_X f_n(x)dx  \leq c_4(1+\log^2 1/\ep)(|B_{n+K}\setminus B_n| + |TT_{n+K}|).\end{equation}
Further it is obvious that $|TT_{n+K}| < 1 - |B_{n+K}| \leq 1- |B_n|$ and since each vertex in $TT_{n+K}$ has a unique, non-empty family inside $B_{n+K}$, we also get that $|TT_{n+K}| \leq |B_{n+K}| \leq |B_n|+c_0/2|B_n|(1-|B_n|) \leq 2|B_n|$ . Hence we can simply write 
\[ |TT_{n+K}| \leq 4|B_n|(1-|B_n|)\] because either $|B_n|$ or $1-|B_n|$ is at least 1/2. We also have by assumption that $|B_{n+K} \setminus B_n| \leq c_0/2|B_n|(1-|B_n|) \leq |B_n|(1-|B_n|)$. Plugging all this into~(\ref{Kstepbound_eq}) we get
\begin{equation}
\int_X f_{n+K}(x) dx - \int_X f_n(x)dx  \leq c_4(1+\log^2 1/\ep)5|B_n|(1-|B_n|),
\end{equation}
and by the assumptions on the small growth of $X_n$ and $B_n$ we can further deduce (assuming $c_4$ is not really small)
\begin{equation} \label{eq:jnKupper}
J(n+K) - J(n) \leq (6c_4(1+\log^2 1/\ep))|B_n|(1-|B_n|) + c_0/2|X_n|(1-|X_n|).
\end{equation} 

On the other hand we can apply~(\ref{jngrowth_eq}) to $n,n+1,\dots, n+K-1$. By the assumption on the small growth of $X_n$ and $B_n$ during this time, $|X_n|(1-|X_n|)$ and $|B_n|(1-|B_n)$ do not change too much either. More precisely we can write for any $n \leq m < n+K$ that $|X_n| \leq |X_m|$ and that  $1- |X_{n+K}| \leq 1- |X_m|$. Also 
\[(1- |X_n|) - (1-|X_{n+K}|) \leq \frac{c_0}{2}|X_n|(1-|X_n|)\] and thus
\[1-|X_{n+K}| \geq (1-\frac{c_0}{2}|X_n|)(1-|X_n) \geq \frac{1-|X_n|}{2}.\]  Putting all this together we get that 
\[|X_m|(1-|X_m|) \geq |X_n|(1-|X_{n+K}|) \geq \frac{1}{2}|X_n|(1-|X_n|),\] and the exact same equation holds for $B_m$.
Now summing~(\ref{jngrowth_eq}) for $n,n+1,\dots, n+K-1$ and using the last inequality, we find that 
\begin{equation}\label{jnKlower_eq}
\frac{K}{2}\left(\frac{c_0|X_n|(1-|X_n|)}{2(d+1)} + \frac{c_0 |B_n|(1-|B_n|)}{4(d+1)^2} \right) \leq  J(n+K) - J(n)
\end{equation}
Now choose $K$ so large that $K > 2(d+ 1)$ and $c_0K > 24c_4(d+1)^2(1+\log^2 1/\ep))$, and we clearly have a contradiction between~(\ref{eq:jnKupper}) and~(\ref{jnKlower_eq}).

\begin{corol}\label{goodmoments_corol}
This implies that for any $n$ either $|X_{n+K}| - |X_n| \geq \frac{c_0}{2}|X_n|(1-|X_n|)$ or $|B_{n+K}|-|B_n| \geq \frac{c_0}{2}|B_n|(1-|B_n|)$.
\end{corol}
 Let us consider the sequences $a_n = |X_{nK+n_0}|, b_n = |B_{nK+n_0}|$. Then Corollary~\ref{goodmoments_corol} implies, using the language of Lemma~\ref{expseq_lemma}, that every $n$ is a good moment for either $a_n$ or $b_n$. We know that $a_ 0 |X_{n_0}| = \ep/2 > \ep/6$. If we also knew that $b_0 = |B_{n_0}| \geq \ep/8$, then by Lemma~\ref{expseq_lemma} we could deduce that for 
  \[k = n_0 + 4K\left \lceil \frac{\log(4/\ep)}{ \log(4/(4-c_0))}\right\rceil\] we have $|X_k| > 1- \ep/8 \geq 1-\ep/2$ or $|B_k| > 1-\ep/8 \geq 1-\ep/2$, either of which implies the existence of an augmenting path. All we need to do to finish the proof of Theorem~\ref{shortalternating_theorem} is to exhibit a not too large $n_0$ for which $|B_{n_0}| > \ep/8$. 

To this end we prove that as long as $|B_n|$ is very small, the size of $X_n$ has to increase rapidly. Obviously $\int f_{n+1}(x) - \int f_n(x) \leq |TE| \leq |B_n|$ since every tough vertex has a nonemtpy family. Hence 
\[J(n+1)-J(n) \leq |X_{n+1}| - |X_n| + \frac{3}{2}|B_{n+1}|.\]  If $|S| \geq 4|B_n|$ then, since clearly $|TB| \leq |B_n|$, we also have $|S| - |TB| \geq 3|B_n|$ and thus by~(\ref{eq:stb}) we get
\[ \frac{c_0}{2(d+1)} |X_n| (1-|X_n|) \leq |X_{n+1}| - |X_n|.\] Then Lemma~\ref{expseq_lemma} implies that this cannot hold for more than 
\[ 2 \left\lceil\frac{\log(1/\ep)}{\log(\frac{1}{1-c_0/4(d+1)})}\right\rceil \] steps. So this is a good choice for $n_0$.
The dependence of $K$ on $\log(1/\ep)$ is quadratic, of $n_0$ linear, hence $k$ is of order $O(\log^3(1/\ep))$, the implied constant only depending on $c_0$ and $d$. This completes the proof of Theorem~\ref{shortalternating_theorem}.

\end{document}